\documentclass{article}

\usepackage{amsmath,amssymb,amsthm}

\newtheorem{theorem}{Theorem}[section]
\newtheorem{definition}[theorem]{Definition}
\newtheorem{proposition}[theorem]{Proposition}
\newtheorem{corollary}[theorem]{Corollary}
\newtheorem{lemma}[theorem]{Lemma}
\newtheorem{fact}[theorem]{Remark}
\newtheorem{exemplu}[theorem]{Example}

\newcommand{\bdfn}{\begin{definition}}
\newcommand{\edfn}{\end{definition}}
\newcommand{\bthm}{\begin{theorem}}
\newcommand{\ethm}{\end{theorem}}
\newcommand{\bprop}{\begin{proposition}}
\newcommand{\eprop}{\end{proposition}}
\newcommand{\bcor}{\begin{corollary}}
\newcommand{\ecor}{\end{corollary}}
\newcommand{\blem}{\begin{lemma}}
\newcommand{\elem}{\end{lemma}}
\newcommand{\bfact}{\begin{fact}}
\newcommand{\efact}{\end{fact}}
\newcommand{\bex}{\begin{exemplu}\begin{rm}}
\newcommand{\eex}{\end{rm}\end{exemplu}}

\def\R{{\mathbb R}}
\def\N{{\mathbb N}}

\newcommand{\lambdaxy}{(1-\lambda)x\oplus\lambda y}

\newcommand{\midxy}{\frac12x\oplus\frac12y}
\newcommand{\midyz}{\frac12y\oplus\frac12z}

\newcommand{\eps}{\varepsilon}
\newcommand{\ol}{\overline}

\newcommand{\be}{\begin{enumerate}}
\newcommand{\ee}{\end{enumerate}}
\newcommand{\bt}{\begin{tabular}}
\newcommand{\et}{\end{tabular}}
\newcommand{\beq}{\begin{equation}}
\newcommand{\eeq}{\end{equation}}
\newcommand{\ba}{\begin{array}} 
\newcommand{\ea}{\end{array}}
\newcommand {\bea} {\begin{eqnarray}}
\newcommand {\eea} {\end {eqnarray}}
\newcommand {\bua} {\begin{eqnarray*}}
\newcommand {\eua} {\end {eqnarray*}}
\newcommand{\se}{\subseteq}
\newcommand{\ds}{\displaystyle}

\newcommand{\Ra}{\Rightarrow}

\newcommand{\limn}{\ds\lim_{n\to\infty}}

\begin{document}

\title{Nonexpansive iterations in uniformly convex $W$-hyperbolic spaces}
\author{Lauren\c tiu Leu\c stean\\[0.2cm] 
\footnotesize Department of Mathematics, Darmstadt University of Technology,\\
\footnotesize Schlossgartenstrasse 7, 64289 Darmstadt, Germany\\[0.1cm]
\footnotesize and\\
\footnotesize Institute of Mathematics "Simion Stoilow'' of the Romanian Academy, \\
\footnotesize Calea Grivi\c tei 21, P.O. Box 1-462, Bucharest, Romania\\[0.1cm]
\footnotesize E-mail: leustean@mathematik.tu-darmstadt.de}
\date{}
\maketitle

\begin{abstract}
We propose the class of uniformly convex $W$-hyperbolic spaces with monotone modulus of uniform convexity ($UCW$-hyperbolic spaces for short) as an appropriate setting for the study of nonexpansive iterations. $UCW$-hyperbolic spaces are a natural generalization both of uniformly convex normed spaces and $CAT(0)$-spaces. Furthermore, we apply proof mining techniques to get effective rates of asymptotic regularity for Ishikawa iterations of nonexpansive self-mappings of closed convex subsets in $UCW$-hyperbolic spaces. These effective results are new even for uniformly convex  Banach spaces.
\end{abstract}
 
\maketitle

\section{Introduction}

In this paper we propose the class of uniformly convex $W$-hyperbolic spaces with monotone modulus of uniform convexity ($UCW$-hyperbolic spaces for short) as an appropriate setting for the study of nonexpansive iterations. This class of geodesic spaces, which will be defined in Section \ref{Ishikawa-section-UC-hyp}, is a natural generalization both of uniformly convex normed spaces and $CAT(0)$-spaces. As we shall see in Section \ref{Ishikawa-section-UC-hyp}, complete $UCW$-hyperbolic spaces have very nice properties. Thus, the intersection of any decreasing sequence of nonempty bounded closed convex subsets is nonempty (Proposition \ref{Ishikawa-CIP-uchyp}) and closed convex subsets are Chebyshev sets (Proposition \ref{uchyp-closed-convex-Chebyshev}). 

The asymptotic center technique, introduced by Edelstein \cite{Ede72,Ede74}, is one of the most useful tools in metric fixed point theory of nonexpansive mappings in uniformly convex Banach spaces, due to the fact that bounded sequences have unique asymptotic centers with respect to closed convex subsets. We prove that this basic property is true for complete $UCW$-hyperbolic spaces too (Proposition \ref{UCW-unique-ac}). The main result of Section \ref{UCW-as-cen-fpp} is Theorem \ref{equiv-T-fpp}, which uses methods involving asymptotic centers  to get, for nonexpansive self-mappings $T:C\to C$ of convex closed subsets of complete $UCW$-hyperbolic spaces, equivalent characterizations of the fact that $T$ has fixed points in terms of boundedness of different iterations associated with $T$. As an immediate consequence of Theorem \ref{equiv-T-fpp}, we obtain a generalization to complete $UCW$-hyperbolic spaces of the well-known Browder-Goehde-Kirk Theorem.

In the second part of the paper, we apply proof mining techniques to give effective rates of asymptotic regularity for Ishikawa iterations of nonexpansive self-mappings of closed convex subsets in $UCW$-hyperbolic spaces. We emphasize that our results are new even for the normed case. By {\em proof mining} we mean the logical analysis of mathematical proofs with the aim of extracting new numerically relevant information hidden in the proofs. We refer to Kohlenbach's book \cite{Koh08-book} for details on proof mining.

If $(X,\|\cdot\|)$ is a normed space, $C\se X$ a nonempty convex subset of $X$  and $T:C\to C$  is nonexpansive, then the {\em Ishikawa iteration} \cite{Ish74} starting with $x\in C$ is defined by 
\beq
x_0:=x, \quad x_{n+1}=(1-\lambda_n)x_n+\lambda_nT((1-s_n)x_n+ s_nTx_n),
\eeq
where $(\lambda_n),(s_n)$ are sequences in $[0,1]$. By letting $s_n=0$ for all $n\in\N$, we get the Krasnoselski-Mann iteration as  a special case.

In Section \ref{Ishikawa-proof-mining}, we consider the important problem of asymptotic regularity associated with the Ishikawa iterations:
\[\lim_{n\to\infty}d(x_n,Tx_n)=0.\]

Our point of departure is the following result, proved by Tan and Xu \cite{TanXu93} for uniformly convex Banach spaces and, recently,  by Dhompongsa and Panyanak \cite{DhoPan08} for $CAT(0)$-spaces.
 
\bprop
Let $X$ be a uniformly convex Banach space or a $CAT(0)$-space, $C\se X$ a nonempty bounded closed convex subset and $T:C\to C$ be nonexpansive.  Assume that $\ds\sum_{n=0}^\infty\lambda_n(1-\lambda_n)$ diverges, $\limsup_n s_n<1$ and $\ds\sum_{n=0}^\infty s_n(1-\lambda_n)$ converges.

Then for all $x\in C$,
\[\limn d(x_n,Tx_n)=0.\]  
\eprop

Using proof mining methods we obtain a quantitative version (Theorem \ref{Ishikawa-quant-as-reg}) of a two-fold generalization of the above proposition:
\begin{itemize}
\item[-] firstly, we  consider $UCW$-hyperbolic spaces;
\item[-] secondly, we assume that $F(T)\ne\emptyset$ instead of assuming the boundedness of $C$.
\end{itemize}
The idea is to combine methods used in \cite{Leu07} to obtain effective rates of asymptotic regularity for Krasnoselski-Mann iterates with the ones used in \cite{Leu07a} to get rates of asymptotic regularity for Halpern iterates.

In this way, we provide for the first time (even for the normed case) effective rates of asymptotic regularity for the Ishikawa iterates, that is rates of convergence of $(d(x_n,Tx_n))$ towards $0$. 

For bounded $C$ (Corollary \ref{Ishikawa-bounded-C}), the rate of asymptotic regularity is uniform in  the nonexpansive mapping $T$ and the starting point $x\in C$ of the iteration, and it depends  on $C$ only via its diameter and on the space $X$ only via the modulus of uniform convexity.

\section{$UCW$-hyperbolic spaces}\label{Ishikawa-section-UC-hyp}

We work in the setting of hyperbolic spaces as introduced by Kohlenbach \cite{Koh05}. In order to distinguish them from Gromov hyperbolic spaces \cite{BriHae99-book} or from other notions of 'hyperbolic space'  which can be found in the literature (see for example \cite{Kir82,GoeKir83,ReiSha90}), we shall call them W-hyperbolic spaces. 

A {\em  $W$-hyperbolic space}  $(X,d,W)$ is a metric space $(X,d)$ together with a convexity mapping $W:X\times X\times [0,1]\to X$ satisfying 
\begin{eqnarray*}
(W1) & d(z,W(x,y,\lambda))\le (1-\lambda)d(z,x)+\lambda d(z,y),\\
(W2) & d(W(x,y,\lambda),W(x,y,\tilde{\lambda}))=|\lambda-\tilde{\lambda}|\cdot 
d(x,y),\\
(W3) & W(x,y,\lambda)=W(y,x,1-\lambda),\\
(W4) & \,\,\,d(W(x,z,\lambda),W(y,w,\lambda)) \le (1-\lambda)d(x,y)+\lambda
d(z,w).
\end {eqnarray*}

The convexity mapping $W$ was first considered by Takahashi in \cite{Tak70}, where a triple $(X,d,W)$ satisfying $(W1)$ is called a {\em convex metric space}. If $(X,d,W)$ satisfies $(W1)-(W3)$, then we get the notion of {\em space of hyperbolic type} in the sense of Goebel and Kirk \cite{GoeKir83}. $(W4)$ was already considered by Itoh \cite{Ito79} under the name `condition III' and it is used by Reich and Shafrir \cite{ReiSha90} and Kirk \cite{Kir82} to define their notions of hyperbolic space. We refer to \cite[p.384-387]{Koh08-book} for a detailed discussion.

The class of $W$-hyperbolic spaces includes normed spaces and convex subsets thereof, the Hilbert ball (see \cite{GoeRei84-book} for a book treatment)  as well as $CAT(0)$-spaces.

If $x,y\in X$ and $\lambda\in[0,1]$, then we use the notation $(1-\lambda)x\oplus \lambda y$ for $W(x,y,\lambda)$. The following holds even for the more general setting of convex metric spaces \cite{Tak70}: for all $x,y\in X$ and  $\lambda\in[0,1]$,
\beq
d(x,\lambdaxy)=\lambda d(x,y),\text{~and~} d(y,\lambdaxy)=(1-\lambda)d(x,y). \label{prop-xylambda}
\eeq
As a consequence, $1x\oplus 0y=x,\,0x\oplus 1y=y$ and $(1-\lambda)x\oplus \lambda x=\lambda x\oplus (1-\lambda)x=x$. 

For all $x,y\in X$, we shall denote by $[x,y]$ the set $\{(1-\lambda)x\oplus \lambda y:\lambda\in[0,1]\}$. Thus, $[x,x]=\{x\}$ and for $x\ne y$, the mapping
\[\gamma_{xy}:[0,d(x,y)]\to\R, \quad \gamma(\alpha)=\left(1-\frac{\alpha}{d(x,y)}\right)x\oplus \frac{\alpha}{d(x,y)}y \]
is a geodesic satisfying $\gamma_{xy}\big([0,d(x,y)]\big)=[x,y]$. That is, any $W$-hyperbolic space is a geodesic space.

A nonempty subset $C\subseteq X$ is {\em convex} if $[x,y]\se C$ for all $x,y\in C$. A nice feature of our setting is that any convex subset is itself a $W$-hyperbolic space with the restriction of $d$ and $W$ to $C$. It is easy to see that open and closed balls are convex. Moreover, using (W4), we get that the closure of a convex subset of a $W$-hyperbolic space is again convex.

If $C$ is a convex subset of $X$, then a function $f:C\to\R$ is said to be {\em convex} if 
\[f\left(\lambdaxy\right)\le (1-\lambda)f(x)+\lambda f(y)\]
for all $x,y\in C,\lambda\in[0,1]$.

One of the most important classes of Banach spaces are the uniformly convex ones, introduced by Clarkson in the 30's \cite{Cla36}. 
Following \cite[p. 105]{GoeRei84-book}, we can define uniform convexity for $W$-hyperbolic spaces too.

A $W$-hyperbolic space $(X,d,W)$ is  {\em uniformly convex} \cite{Leu07} if for 
any $r>0$ and any $\varepsilon\in(0,2]$ there exists $\delta\in(0,1]$ such that 
for all $a,x,y\in X$,
\begin{eqnarray}
\left.\begin{array}{l}
d(x,a)\le r\\
d(y,a)\le r\\
d(x,y)\ge\varepsilon r
\end{array}
\right\}
& \quad \Rightarrow & \quad d\left(\frac12x\oplus\frac12y,a\right)\le (1-\delta)r. \label{Ishikawa-uc-def}
\end{eqnarray}
A mapping $\eta:(0,\infty)\times(0,2]\rightarrow (0,1]$ providing such a
$\delta:=\eta(r,\varepsilon)$ for given $r>0$ and $\varepsilon\in(0,2]$ is called a {\em modulus of uniform convexity}. We call $\eta$ {\em monotone} if it decreases with $r$ (for a fixed $\eps$).

\blem\label{eta-prop-1}\cite{Leu07,KohLeu08a}
Let $(X,d,W)$ be a uniformly convex $W$-hyperbolic space and $\eta$ be a modulus of uniform convexity. Assume that $r>0,\varepsilon\in(0,2], a,x,y\in X$ are such that 
\[d(x,a)\le r,\,\,d(y,a)\le r \text{~and~} d(x,y)\ge\eps r.\]
 Then for any $\lambda\in[0,1]$,
\be
\item\label{Groetsch-eta} $\ds d(\lambdaxy,a)\le  \big(1-2\lambda(1-\lambda)\eta(r,\varepsilon)\big)r$; 
\item\label{eta-monotone-eps} for any $\psi\in (0,2]$ such that $\psi\le\eps$, 
\[\ds d(\lambdaxy,a)\le  \big(1-2\lambda(1-\lambda)\eta(r,\psi)\big)r\,;\]
\item  \label{eta-s-geq-r} for any $s\geq r$, 
\[d(\lambdaxy,a) \le \left(1-2\lambda(1-\lambda)\eta\left(s,\frac{\eps r}{s}\right)\right)s\,;\]
\item\label{eta-monotone-s-geq-r}if $\eta$ is monotone, then for any $s\geq r$, 
\[d(\lambdaxy,a) \le \left(1-2\lambda(1-\lambda)\eta\left(s,\eps\right)\right)r\,.\]
\ee
\elem
\begin{proof}
(i) is a generalization to our setting of a result due to Groetsch \cite{Gro72}. We refer to \cite[Lemma 7]{Leu07} for the proof.\\
(ii),(iii) are immediate; see \cite[Lemma 2.1]{KohLeu08a}.\\
(iv) Use (i) and the fact that $\eta(r,\eps)\geq \eta(s,\eps)$, hence $1-2\lambda(1-\lambda)\eta(r,\varepsilon)\leq 1-2\lambda(1-\lambda)\eta(s,\varepsilon)$.
\end{proof}

We shall refer to  uniformly convex $W$-hyperbolic spaces with a monotone modulus of uniform convexity as {\em $UCW$-hyperbolic spaces}. It turns out  \cite{Leu07} that $CAT(0)$-spaces  are   $UCW$-hyperbolic spaces with modulus of uniform convexity $\ds\eta(r,\varepsilon)=\varepsilon^2/8$ quadratic in $\eps$. Thus, $UCW$-hyperbolic spaces are a natural generalization of both uniformly convex normed spaces and $CAT(0)$-spaces.

For the rest of this section, $(X,d,W)$ is a complete $UCW$-hyperbolic space and $\eta$ is a monotone modulus of uniform convexity. 

\bprop\label{Ishikawa-CIP-uchyp}\cite[Proposition 2.2]{KohLeu08a}
The intersection of any decreasing sequence of nonempty bounded closed convex subsets of $X$ is nonempty.
\eprop

The next proposition, inspired by \cite[Proposition 2.2]{GoeRei84-book}, is essential for what it follows. 

\bprop\label{uc-hyp-basic-prop-as-center}
Let $C$ be a nonempty closed convex subset of $X$, $f:C\to[0,\infty)$ be convex and lower semicontinuous. Assume moreover that for all sequences $(x_n)$ in $C$,
\[\limn d(x_n,a)=\infty \text{~for some~} a\in X \text{~implies~} \limn f(x_n)=\infty.\]
Then $f$ attains its minimum on $C$.  If, in addition, 
\[f\left(\midxy\right) < \max\{f(x),f(y)\}\]
for all $x\neq y$, then $f$ attains its minimum at exactly one point.
\eprop
\begin{proof} 
Let $\alpha$ be the infimum of $f$ on $C$ and define
\[C_n:=\left\{x\in C\mid f(x)\leq \alpha+\frac1n\right\}\]
for all $n\in\N$. It is easy to see that we can apply Proposition \ref{Ishikawa-CIP-uchyp} to the sequence $(C_n)_{n\in\N}$ to get the existence of $x^\star\in \bigcap_{n\in\N}C_n$. It follows that $\ds f(x^\star)\leq \alpha+\frac1n$ for all $n\geq 1$, hence $f(x^\star)\leq \alpha$. Since $\alpha$ is the infimum of $f$, we can conclude that $f(x^\star)=\alpha$, that is $f$ attains its minimum on $C$. The second part of the conclusion is immediate. If $f$ attains its minimum at two points $x^\star\neq y^\star$, then $\frac12 x^\star\oplus \frac12 y^\star\in C$, since $C$ is convex,  and $\ds f\left(\frac12 x^\star\oplus \frac12 y^\star\right) <\max\{f( x^\star), f(y^\star)\}=\alpha$, which is a contradiction.
\end{proof}

Let us recall that a subset $C$ of a metric space $(X,d)$ is called a {\em Chebyshev set} if to each point $x\in X$ there corresponds a unique point $z\in C$ such that $d(x,z)= d(x,C)(=\inf\{d(x,y)\mid y\in C\})$. If $C$ is a Chebyshev set, one can define the {\em nearest point projection} $P:X\to C$ by assigning $z$ to $x$.

\bprop\label{uchyp-closed-convex-Chebyshev}
Every  nonempty closed convex subset  $C$ of $X$ is a Chebyshev set.
\eprop
\begin{proof} 
Let $x\in X$ and define $f:C\to [0,\infty), \quad f(y)=d(x,y)$. 
Then $f$ is continuous, convex (by (W1)), and for any sequence $(y_n)$ in $C$, $\limn d(y_n,a)=\infty$ for some $a\in X$ implies $\limn f(y_n)= \infty$, since $f(y_n)=d(x,y_n)\geq d(y_n,a)-d(x,a)$. Moreover, let $y\neq z\in C$ and denote $M:=\max\{f(y),f(z)\}>0$. Then 
\[d(x,y), d(x,z)\leq M\,\,\text{~and~} d(y,z)\geq \eps\cdot M,\]
where $\ds\eps:=\frac{d(y,z)}M$ and $\ds 0<\eps\leq \frac{d(x,y)+d(x,z)}M\leq 2 $. Hence, by uniform convexity it follows that
\[d\left(\midyz,x\right)\leq (1-\eta(M,\eps))\cdot M<M.\]
Thus, $f$ satisfies all the hypotheses of Proposition \ref{uc-hyp-basic-prop-as-center}, so we can apply it to conclude that $f$ has a unique minimum. Hence, $C$ is a Chebyshev set. 
\end{proof}

\section{Asymptotic centers and fixed point theory of nonexpansive mappings}\label{UCW-as-cen-fpp} 

In the sequel, we recall basic facts about asymptotic centers. We refer to \cite{Ede72,Ede74,GoeRei84-book} for all the unproved results.

Let $(X,d)$ be a metric space, $(x_n)$ be a bounded sequence in $X$ and $C\se X$ be a nonempty subset of $X$. We define the following functionals:

\bua
r_m(\cdot, (x_n)):X\to [0,\infty), \quad r_m(y,(x_n))&=&\sup\{d(y,x_n)\mid n\geq m\}\\
&& \text{for~}m\in\N,\\
r(\cdot, (x_n)):X\to[0,\infty), \quad r(y,(x_n)) &=&\limsup_{n}d(y,x_n)=\inf_{m}r_m(y,(x_n)) \label{def-r-xn}\\
&=&\lim_{m\to\infty}r_m(y,(x_n)).
\eua

The following lemma collects some basic properties of the above functionals. 

\blem\label{prop-rm-r}
Let $y\in X$.
\be
\item $r_m(\cdot, (x_n))$ is nonexpansive for all $m\in \N$;
\item $r(\cdot, (x_n))$ is continuous and $r(y,(x_n))\to\infty$ whenever $d(y,a)\to\infty$ for some $a\in X$;
\item $r(y, (x_n))=0$ if and only if $\limn x_n=y$;
\item if $(X,d,W)$ is a convex metric space and $C$ is convex, then $r(\cdot, (x_n))$ is a convex function.
\ee
\elem

The {\em asymptotic radius of } $(x_n)$ {\em with respect to} $C$ is defined by
\[r(C,(x_n))  = \inf\{r(y,(x_n))\mid y\in C\}.\]
The  {\em asymptotic radius} of $(x_n)$, denoted by $r((x_n))$, is the asymptotic radius of $(x_n)$ with respect to $X$, that is $r((x_n))=r(X,(x_n))$.

A point $c\in C$ is said to be an {\em asymptotic center} of $(x_n)$ {\em with respect to } $C$ if
\[r(c,(x_n))= r(C,(x_n))=\min\{r(y,(x_n))\mid y\in C\}.\]
We denote with $A(C,(x_n))$ the set of asymptotic centers of $(x_n)$ with respect to $C$. When $C=X$, we call $c$ an {\em asymptotic center} of $(x_n)$ and we use the notation $A((x_n))$ for $A(X,(x_n))$. 

The following lemma, inspired by \cite[Theorem 1]{Ede74}, turns out to be very useful in the following. 
 
\blem\label{useful-unique-as-center} 
Let $(x_n)$ be a bounded sequence in $X$ with $A(C,(x_n))=\{c\}$ and  $(\alpha_n),(\beta_n)$ be real sequences such that $\alpha_n\geq 0$ for all $n\in\N$, $\limsup_n \alpha_n\leq 1$ and $\limsup_n \beta_n\leq 0$.\\
 Assume that $y\in C$  is such that there exist $p,N\in\N$ satisfying 
\[\forall n\ge N\bigg(d(y,x_{n+p})\leq \alpha_nd(c,x_n)+\beta_n\bigg).\] 
Then $y=c$.
\elem
\begin{proof}
We have that 
\bua
r(y,(x_n))&=&\!\!\!\limsup_n d(y,x_n)= \limsup_n d(y,x_{n+p})\leq \limsup_n \big(\alpha_nd(c,x_n)+\beta_n\big)\\
&\leq & \!\!\! \limsup_n\alpha_n\cdot \limsup_n d(c,x_n)+\limsup_n\beta_n \leq \limsup_n d(c,x_n)\\
&=& \!\!\!r(c,(x_n)).
\eua
 Since $c$ is unique with the property that $r(c,(x_n))=\min\{r(z,(x_n))\mid z\in C\}$, we must have $y=c$.
\end{proof}

In general, the set $A(C,(x_n))$ of asymptotic centers of a bounded sequence $(x_n)$ with respect to $C\se X$ may be empty or even contain infinitely many points. 

The following result shows that in the case of complete $UCW$-hyperbolic spaces, the situation is as nice as for uniformly convex Banach spaces (see, for example, \cite[Theorem 4.1]{GoeRei84-book}).

\bprop\label{UCW-unique-ac}
Let $(X,d,W)$ be a complete $UCW$-hyperbolic space. Every bounded sequence $(x_n)$ in $X$ has a unique asymptotic center with respect to any nonempty closed convex subset $C$ of $X$.
\eprop
\begin{proof}
Let $\eta$ be a monotone modulus of uniform convexity. We apply Proposition \ref{uc-hyp-basic-prop-as-center} to show that the function $r(\cdot, (x_n)):C\to [0,\infty)$ attains its minimum at exactly one point. By Lemma \ref{prop-rm-r}, it remains to prove that 
\[ r\left(\midyz,(x_n)\right) < \max\{r(y,(x_n)),r(z,(x_n))\} \quad \text{whenever~} y,z\in C, y\neq z.\]
Let $M:=\max\{r(y,(x_n)),r(z,(x_n))\}>0$. For every $\eps\in(0,1]$ there exists $N$ such that $d(y,x_n),d(z,x_n)\leq M+\eps\le M+1$ for all  $n\geq N$. Moreover, $\ds d(y,z) =\frac{d(y,z)}{M+\eps}\cdot (M+\eps)\geq \frac{d(y,z)}{M+1}\cdot (M+\eps)$.
Thus, we can apply Lemma \ref{eta-prop-1}. (\ref{eta-monotone-s-geq-r}) to get that for all $n\geq N$,
\bua
d\left(\midyz,x_n\right) & \leq & \left(1-\eta\left(M+1,\frac{d(y,z)}{M+1}\right)\right)(M+\eps),
\eua
hence
\[r\left(\midyz,(x_n)\right)\leq  \left(1-\eta\left(M+1,\frac{d(y,z)}{M+1}\right)\right)(M+\eps).\]
By letting $\eps\to 0$, it follows that 
\[r\left(\midyz,(x_n)\right)\leq  \left(1-\eta\left(M+1,\frac{d(y,z)}{M+1}\right)\right)\cdot M\,\,<\,M.\]
This completes the proof.
\end{proof}

Let $T:C\to C$. We shall denote with $F(T)$ the set of fixed points of $T$. For any $x\in C$ and any $b,\eps>0$ we shall use the notation
\[Fix_{\varepsilon}(T,x,b)=\{y\in C\mid d(y,x)\leq b \text{~and~} d(y,Ty)<\eps\}.\]
If $Fix_{\varepsilon}(T,x,b)\ne \emptyset$ for all $\eps>0$, we say that $T$ {\em has approximate fixed points} in a $b$-neighborhood of $x$ . 

\blem
The following are equivalent.
\be
\item there exists a bounded sequence $(x_n)$ in $C$ such that $\limn d(x_n,Tx_n)=0$;
\item for all $x\in C$ there exists  $b>0$ such that $T$ has approximate fixed points in a $b$-neighborhood of $x$;
\item there exist $x\in C$ and $b>0$ such that $T$ has approximate fixed points in a $b$-neighborhood of $x$.
\ee 
\elem
\begin{proof}
$(i)\Ra(ii)$ Take as $b$ any bound on $(d(x,x_n))$.\\
$(ii)\Ra(iii)$ Obviously.\\
$(iii)\Ra(i)$ Let $x\in C$ and $b>0$ be such that $Fix_{\varepsilon}(T,x,b)\ne \emptyset$ for all $\eps>0$. Apply this with $\ds\eps:=\frac1n$ to get $x_n\in C$ satisfying (i).
\end{proof}

In the sequel, we assume that $(X,d,W)$ is a $W$-hyperbolic space, $C\se X$  is convex and $T:C\to C$ is {\em nonexpansive}, that is 
\[d(Tx,Ty)\leq d(x,y)\]
for all $x,y\in C$. For any $\lambda\in(0,1]$, the {\em averaged mapping} $T_\lambda$ is defined by
\[T_\lambda:C\to C, \quad T_\lambda(x)=(1-\lambda)x\oplus\lambda Tx.\]
It is easy to see that $T_\lambda$ is also nonexpansive and that $F(T)=F(T_\lambda)$. 

The {\em Krasnoselski iteration} \cite{Kra55, Sch57} $(x_n)$ starting with $x\in C$ is defined as the Picard iteration $\big(T_\lambda^n(x)\big)$ of $T_\lambda$, that is
\begin{equation}
x_0:=x, \quad x_{n+1}:=(1-\lambda)x_n \oplus\lambda Tx_n. \label{Ishikawa-K-iteration-def-hyp}
\end{equation}
By allowing general sequences  $(\lambda_n)$ in $[0,1]$, we get the {\em Krasnoselski-Mann iteration} \cite{Man53} (called {\em segmenting Mann iterate} in \cite{Gro72})  $(x_n)$ starting with $x\in C$: 
\begin{equation}
x_0:=x, \quad x_{n+1}:=(1-\lambda_n)x_n \oplus\lambda_n Tx_n. \label{Ishikawa-KM-iteration-def-hyp}
\end{equation}

The following lemma collects some known properties of Krasnoselski-Mann iterates in $W$-hyperbolic spaces. For the sake of completeness we prove them here.

\begin{lemma}\label{useful-KM-x-y}
Let $(x_n),(y_n)$ be the Krasnoselski-Mann iterations starting with $x,y\in C$. Then 
\be
\item\label{KM-xn-yn-decreasing} $(d(x_n,y_n))$ is decreasing;
\item\label{KM-xn-fp-decreasing} if $p$ is a fixed point of $T$, then $(d(x_n,p))$ is decreasing;
\item\label{KM-useful-xn+1-Ty} $d(x_{n+1},Ty)\le d(x_n,y)+(1-\lambda_n)d(y,Ty)$ for all $n\in\N$,
\ee
\end{lemma}
\begin{proof}
\bua
d(x_{n+1},y_{n+1})&\leq&(1-\lambda_n)d(x_n,y_n)+\lambda_n d(Tx_n,Ty_n)\quad \text{by (W4)}\\
&\le& d(x_n,y_n), \quad \text{since T is nonexpansive},\\
d(x_{n+1},p)&\le & (1-\lambda_n)d(x_n,p)+\lambda_nd(Tx_n,p) \quad \text{by (W1)} \\
&=& (1-\lambda_n)d(x_n,p)+\lambda_nd(Tx_n,Tp)\\
&\leq&  (1-\lambda_n)d(x_n,p)+\lambda_n d(x_n,p)= d(x_n,p),\\
d(x_{n+1},Ty)&\le & (1-\lambda_n)d(x_n,Ty)+\lambda_n d(Tx_n,Ty)\quad \text{by (W1)}\\
&\le &(1-\lambda_n)d(x_n,y)+(1-\lambda_n)d(Ty,y)+\lambda_nd(x_n,y)\\
&\le &d(x_n,y)+(1-\lambda_n)d(Ty,y).
\end {eqnarray*}
\end{proof}

We can prove now the main theorem of this section.

\bthm\label{equiv-T-fpp}
Let $(X,d,W)$ be a complete $UCW$-hyperbolic space, $C\se X$ a nonempty convex closed subset and $T:C\to C$ be nonexpansive.\\
 The following are equivalent.
\be
\item\label{T-f} $T$ has fixed points;
\item\label{un-bounded-Tun-un} there exists a bounded sequence $(u_n)$ in $C$ such that $\limn d(u_n,Tu_n)=0$;
\item\label{P-bounded-exists-x} the sequence $(T^nx)$ of Picard iterates is bounded for some $x\in C$;
\item\label{P-bounded-forall-x} the sequence $(T^nx)$ of Picard iterates is bounded for all $x\in C$;
\item\label{KM-bounded-exists-x-lambdan-ABC} the Krasnoselski-Mann  iteration $(x_n)$ is bounded for some $x\in C$ and  for $(\lambda_n)$ in $[0,1]$ satisfying one of the following conditions:
\be
\item $\lambda_n=\lambda\in(0,1]$;
\item $\limn\lambda_n=1$;
\item $\limsup_n \lambda_n<1$ and $\ds \sum_{n=0}^\infty\lambda_n$ diverges;
\ee
\item\label{KM-bounded-forall-x-lambdan} the Krasnoselski-Mann  iteration $(x_n)$ is bounded for all $x\in C$ and all $(\lambda_n)$ in $[0,1]$.
\ee
\ethm
\begin{proof}
$(\ref{T-f})\Ra(\ref{un-bounded-Tun-un})$ Let $p$ be a fixed point of $T$ and define $u_n:=p$ for all $n\in\N$.
$(\ref{un-bounded-Tun-un})\Ra(\ref{T-f})$ By Proposition \ref{UCW-unique-ac}, $(u_n)$ has a unique asymptotic center $c$ with respect to $C$. We get that for all $n\in\N$,
\bua
d(Tc,u_n)\leq d(Tc,Tu_n)+d(Tu_n,u_n)\leq d(c,u_n)+d(Tu_n,u_n).
\eua
We can apply now Lemma \ref{useful-unique-as-center}  with $y:=Tc$ and $p:=N:=0$, $\alpha_n:=1,\beta_n:=d(u_n,Tu_n)$ to get that $Tc=c$.\\
$(\ref{T-f})\Ra(\ref{P-bounded-exists-x})$ If $p$ is a fixed point of $T$, then $T^np=p$ for all $n\in\N$.\\
$(\ref{P-bounded-exists-x})\Ra(\ref{P-bounded-forall-x})$ Apply the fact that, since $T$ is nonexpansive, $d(T^nx,T^ny)\leq d(x,y)$ for all $x,y\in C$.\\
$(\ref{P-bounded-forall-x})\Ra(\ref{T-f})$ Let $c\in C$ be the unique asymptotic center of $(T^nx)$. Then
for all $n\in\N$,
\bua
d(Tc,T^{n+1}x)\leq d(c,T^nx),
\eua
hence we can apply Lemma \ref{useful-unique-as-center}  with $y:=Tc,x_n:=T^nx$ and $p:=1,N:=0$, $\alpha_n:=1,\beta_n:=0$ to get that $Tc=c$.\\
$(\ref{T-f})\Ra(\ref{KM-bounded-forall-x-lambdan})$ Let $p$ be a fixed point of $T$. Then for any $x\in C,(\lambda_n)$ in $[0,1]$, the sequence $(d(x_n,p)$ is decreasing, hence bounded from above by $d(x,p)$.\\
$(\ref{KM-bounded-forall-x-lambdan})\Ra(\ref{KM-bounded-exists-x-lambdan-ABC})$ Obviously.\\
$(\ref{KM-bounded-exists-x-lambdan-ABC})\Ra(\ref{T-f})$ 
\be
\item[(a)] If $\lambda_n=\lambda\in(0,1]$, then $(x_n)$ is the Krasnoselski iteration, hence the Picard iteration $T_\lambda^n(x)$ of the nonexpansive mapping $T_\lambda$. Apply now $(\ref{P-bounded-exists-x})\Ra(\ref{T-f})$ and the fact that $F(T)=F(T_\lambda)$ to get that $T$ has fixed points. 
\item[(b)] 
Assume now that $\limn \lambda_n=1$ and let $c\in C$ be the asymptotic center of $(x_n)$. By Lemma \ref{useful-KM-x-y}.(\ref{KM-useful-xn+1-Ty}), we get that
\bua
d(Tc,x_{n+1})\leq d(c,x_n)+(1-\lambda_n)d(c,Tc).
\eua
Apply now Lemma \ref{useful-unique-as-center} with $y:=Tc$ and $p:=1,N:=0$, $\alpha_n:=1,\beta_n:=(1-\lambda_n)d(c,Tc)$ to get that $Tc=c$.
\item[(c)] If $(\lambda_n)$ is bounded away from $1$ and divergent in sum,  $\lim d(x_n,Tx_n)=0$ by 
\cite[Theorem 3.21]{KohLeu03}, proved even for $W$-hyperbolic space. Hence (ii) holds.
\ee
\end{proof}

As an immediate consequence we obtain the generalization to complete $UCW$-hyperbolic spaces of the well-known Browder-Goehde-Kirk Theorem.

\bcor\label{BGK-uc-hyp-monotone}
Let $(X,d,W)$ be a complete $UCW$-hyperbolic space, $C\se X$  a nonempty bounded convex closed subset and $T:C\to C$ be nonexpansive. Then $T$ has fixed points.\ecor

\section{Rates of asymptotic regularity for the Ishikawa iterates}\label{Ishikawa-proof-mining}

\noindent Let $(X,d,W)$ be a W-hyperbolic space, $C\se X$ a nonempty convex subset of $X$  and $T:C\to C$ be nonexpansive. 

As in the case of normed spaces, we can define the {\em Ishikawa iteration} \cite{Ish74} starting with $x\in C$ by 
\beq
x_0:=x, \quad x_{n+1}=(1-\lambda_n)x_n\oplus \lambda_nT((1-s_n)x_n\oplus s_nTx_n),
\eeq
where $(\lambda_n),(s_n)$ are sequences in $[0,1]$. By letting $s_n=0$ for all $n\in\N$, we get the Krasnoselski-Mann iteration as  a special case.

We shall use the following notations 
\[y_n := (1-s_n)x_n\oplus s_nTx_n\]
and
\[T_n:C\to C, \quad T_n(x)= (1-\lambda_n)x\oplus \lambda_nT((1-s_n)x\oplus s_nTx).\]
Then 
\[x_{n+1}=(1-\lambda_n)x_n\oplus \lambda_nTy_n=T_nx_n\]
and it is easy to see that $F(T)\se F(T_n)$ for all $n\in\N$.

Before proving the main technical lemma, we give some  basic properties of Ishikawa iterates, which hold even in the very general setting of $W$-hyperbolic spaces. Their proofs follow closely the ones of the corresponding properties in uniformly convex Banach spaces \cite{TanXu93} or $CAT(0)$-spaces \cite{DhoPan08}, but, for the sake of completeness, we include the details. 

\blem\label{Ishikawa-useful}
\be
\item 
\bea
d(x_n,x_{n+1})=\lambda_nd(x_n,Ty_n), \quad d(Ty_n,x_{n+1})=(1-\lambda_n)d(x_n,Ty_n),\label{d-xn-Tyn-xn+1}\\
d(y_n,x_n)= s_nd(x_n,Tx_n), \quad d(y_n,Tx_n)=(1-s_n)d(x_n,Tx_n),\label{d-xn-yn-Txn}\\
(1-s_n)d(x_n,Tx_n) \leq d(x_n,Ty_n)\leq (1+s_n)d(x_n,Tx_n),\label{d-xn-Txn-d-xn-Tyn}\\
d(y_n,Ty_n)\leq d(x_n,Tx_n),\label{d-xn-Txn-d-yn-Tyn}\\
d(x_{n+1},Tx_{n+1})\leq (1+2s_n(1-\lambda_n))d(x_n,Tx_n). \label{Ishikawa-as-reg-base-ineq}
\eea
\item $T_n$ is nonexpansive for all $n\in\N$;
\item\label{Ishikawa-useful-p} For all  $p\in F(T)$, the sequence $(d(x_n,p))$ is decreasing  and 
\[d(y_n,p)\leq d(x_n,p)\quad \text{and}\quad d(x_n,Ty_n),d(x_n,Tx_n)\leq 2d(x_n,p).\]
\ee
\elem
\begin{proof}
\be
\item 
(\ref{d-xn-Tyn-xn+1}) and (\ref{d-xn-yn-Txn}) follow from (\ref{prop-xylambda}).\bua
d(x_n,Tx_n)&\leq &  d(x_n,Ty_n)+d(Ty_n,Tx_n)\leq d(x_n,Ty_n)+d(x_n,y_n)\\
&= & d(x_n,Ty_n)+s_nd(x_n,Tx_n) \,\,\text{by (\ref{d-xn-yn-Txn})},
\eua
hence $(1-s_n)d(x_n,Tx_n) \leq d(x_n,Ty_n)$.
\bua
d(x_n,Ty_n) &\leq & d(x_n,Tx_n)+d(Tx_n,Ty_n)\leq d(x_n,Tx_n)+T(x_n,y_n)\\
&=&(1+s_n)d(x_n,Tx_n) \,\,\text{by (\ref{d-xn-yn-Txn})}.\\
d(y_n,Ty_n)&\leq & (1-s_n)d(x_n,Ty_n)+s_nd(Tx_n,Ty_n)\quad \text{by (W1)}\\
&\leq& (1-s_n)(1+s_n)d(x_n,Tx_n)+s_nd(x_n,y_n)\quad \text{by (\ref{d-xn-Txn-d-xn-Tyn})}\\
&=&d(x_n,Tx_n)\quad \text{by (\ref{d-xn-yn-Txn})}.
\eua
Let us prove now (\ref{Ishikawa-as-reg-base-ineq}). First, let us remark that
\bua
d(x_n,Tx_{n+1}) &\leq & d(x_n,x_{n+1})+d(x_{n+1},Tx_{n+1 })\\
&=& \lambda_nd(x_n,Ty_n)+d(x_{n+1},Tx_{n+1 })\quad  \text{by (\ref{d-xn-Tyn-xn+1})}\\
\text{and}\\
d(y_n,x_{n+1})&\leq & (1-\lambda_n)d(y_n,x_n)+\lambda_nd(y_n,Ty_n)\quad \text{by (W1)}.
\eua
Moreover,
\bua
d(x_{n+1},Tx_{n+1})&\leq & (1-\lambda_n)d(x_n,Tx_{n+1})+\lambda_nd(Ty_n,Tx_{n+1})\, \text{by (W1)}\\
&\leq & (1-\lambda_n)\big[d(x_n,x_{n+1})+d(x_{n+1},Tx_{n+1 })\big]+\\
&& +\lambda_nd(y_n,x_{n+1})\\
&\leq &(1-\lambda_n)d(x_{n+1},Tx_{n+1 })+(1-\lambda_n)\lambda_nd(x_n,Ty_n)+\\
&& +\lambda_n(1-\lambda_n)d(y_n,x_n)+\lambda_n^2d(y_n,Ty_n)\\
&& \text{by (\ref{d-xn-Tyn-xn+1}) and (W1), hence}\\
d(x_{n+1},Tx_{n+1})&\leq & (1-\lambda_n)d(x_n,Ty_n)+(1-\lambda_n)d(y_n,x_n)+\\
&& +\lambda_nd(y_n,Ty_n)\\
&\leq & (1-\lambda_n)(1+s_n)d(x_n,Tx_n)+(1-\lambda_n)s_nd(x_n,Tx_n)\\
&&+\lambda_nd(x_n,Tx_n)\quad \text{by (\ref{d-xn-Txn-d-xn-Tyn}), (\ref{d-xn-yn-Txn}) and (\ref{d-xn-Txn-d-yn-Tyn})}\\
&=& (1+2s_n(1-\lambda_n))d(x_n,Tx_n).
\eua
\item 
\bua
d(T_nx,T_ny) &\leq & \!\!\!\lambda_nd\big(T((1-s_n)x\oplus s_nTx,\, T((1-s_n)y\oplus s_nTy)\big)+\\
&& +(1-\lambda_n)d(x,y)\\
&\leq & \!\!\!(1-\lambda_n)d(x,y)+ \lambda_n\big[(1-s_n)d(x,y)+s_nd(Tx,Ty)\big]\\
&& \text{by (W4)}\\
&\leq & \!\!\!(1-\lambda_n)d(x,y)+\lambda_n\big[(1-s_n)d(x,y)+s_nd(x,y)\big]\\
&=& \!\!\!d(x,y).
\eua
\item
\bua
d(x_{n+1},p) &=& d(T_nx_n,T_np)\leq d(x_n,p),\\
d(y_n,p) &\leq & (1-s_n)d(x_n,p)+s_n d(Tx_n,p)\\
&=& (1-s_n)d(x_n,p)+s_n d(Tx_n,Tp)\leq  d(x_n,p),\\
d(x_n,Tx_n) &\leq & d(x_n,p)+d(Tx_n,p)\leq 2d(x_n,p),\\
d(x_n,Ty_n) &\leq & d(x_n,p)+d(Ty_n,p)\leq d(x_n,p)+d(y_n,p)\leq 2d(x_n,p).
\eua
\ee
\end{proof}

\begin{lemma}{\bf (Main technical lemma)}\label{Ishikawa-main-technical-lemma}\\Assume that $(X,d,W)$ is a $UCW$-hyperbolic space with a monotone modulus of uniform convexity $\eta$ and $p\in F(T)$. Let $x\in C,n\in\N$.
\be
\item \label{IMTL-eta} If $\gamma,\beta,\tilde{\beta},a>0$ are such that 
\bua
\gamma\le d(x_n,p)\le \beta,\tilde{\beta} \quad  \text{and~~} \quad  a\le d(x_n,Ty_n),\eua
then 
\bua
d(x_{n+1},p)&\le& d(x_n,p) - 2\gamma\lambda_n(1-\lambda_n)\eta\left(\tilde{\beta},\frac{a}{\beta}\right).
\eua
\item \label{IMTL-tilde-eta} Assume moreover that $\eta$ can be written as $\eta(r,\varepsilon)=\varepsilon\cdot\tilde{\eta}(r,\varepsilon)$ such that $\tilde{\eta}$ increases with $\varepsilon$ (for a fixed $r$). If $\delta,a>0$ are such that 
\bua
d(x_n,p)\le \delta \quad  \text{and~}\quad  a\le d(x_n,Ty_n),
\eua
then
\bua
d(x_{n+1},p)&\le& d(x_n,p)-2a\lambda_n(1-\lambda_n)\tilde{\eta}\left(\delta,\frac{a}{\delta}\right).
\eua
\ee
\end{lemma}
\begin{proof}
\be
\item First, let us remark that, using Lemma \ref{Ishikawa-useful}.(\ref{Ishikawa-useful-p})
\begin{eqnarray*}
d(Ty_n,p)=d(Ty_n,Tp)\le d(y_n,p)\le d(x_n,p)\le \beta,\tilde{\beta},  \\
d(x_n,Ty_n)\ge a= \left(\frac{a}{\beta}\right)\cdot\beta\ge \left(\frac{a}{\beta}\right)\cdot d(x_n,p), \text{~and~}\\
0<a\le d(x_n,Ty_n)\le 2d(x_n,p)\le  2\beta, \text{~so~}\frac{a}{\beta}\in(0,2].
\end {eqnarray*}
Thus, we can apply Lemma \ref{eta-prop-1}.(\ref{eta-monotone-s-geq-r}) with $\ds r:=d(x_n,p), s:=\tilde{\beta},\eps:=\frac{a}{\beta}$ to obtain
\begin{eqnarray*}
d(x_{n+1},p)&=&d((1-\lambda_n)x_n\oplus\lambda_nTy_n,p)\\
&\le& \left(1- 2\lambda_n(1-\lambda_n)\eta\left(\tilde{\beta},\frac{a}{\beta}\right)\right)d(x_n,p)\\
&=& d(x_n,p) - 2\lambda_n(1-\lambda_n)\eta\left(\tilde{\beta},\frac{a}{\beta}\right)d(x_n,p)\\
&\le&d(x_n,p) - 2\gamma\lambda_n(1-\lambda_n)\eta\left(\tilde{\beta},\frac{a}{\beta}\right),
\end {eqnarray*} 
since $d(x_n,p)\ge \gamma$ by hypothesis.
\item Since, by Lemma \ref{Ishikawa-useful}.(\ref{Ishikawa-useful-p}), $0<a\le d(x_n,Ty_n)\le 2d(x_n,p)$, we can apply (i)  with $\gamma:=\beta:=d(x_n,p)>0$ and $\tilde{\beta}:=\delta$ to get that 
\begin{eqnarray*}
d(x_{n+1},p)&\le& d(x_n,p)- 2d(x_n,p)\lambda_n(1-\lambda_n)\eta\left(\delta,\frac{a}{d(x_n,p))}\right)\\
&=& d(x_n,p)- 2a\lambda_n(1-\lambda_n)\tilde{\eta}\left(\delta,\frac{a}{d(x_n,p)}\right)\\
&\le& d(x_n,p)- 2a\lambda_n(1-\lambda_n)\tilde{\eta}\left(\delta,\frac{a}{\delta}\right),
\end {eqnarray*}
since $\displaystyle \frac{a}\delta\le \frac{a}{d(x_n,p)}$ and $\tilde{\eta}$ increases with $\varepsilon$ by hypothesis.
\ee
\end{proof}

We recall some terminology.  Let $(a_n)_{n\geq 0}$ be a sequence of real numbers. A {\em rate of divergence} of a divergent series $\ds \sum_{n=0}^\infty a_n$ is  a function $\theta:\N\to\N$ satisfying  $\ds\sum_{i=0}^{\theta(n)}a_i \geq n$ for all $n\in\N$. \\
If $\limn a_n=a\in\R$, then a function $\gamma:(0,\infty)\to\N$ is called 
\be
\item[-] a {\em Cauchy modulus} of $(a_n)$ if $|a_{\gamma(\eps)+n}-a_{\gamma(\eps)}|$ for all $\eps>0,n\in\N$;
\item[-] a {\em rate of convergence} of $(a_n)$ if $|a_{\gamma(\eps)+n}-a|<\eps$ for all $\eps>0,n\in\N$.
\ee
A {\em Cauchy modulus} of a convergent series $\ds \sum_{n=0}^\infty a_n$ is a Cauchy modulus of the sequence $(s_n)$ of partial sums, $s_n:=\ds \sum_{i=0}^n a_i$.

\bprop\label{Ishikawa-liminf-xn-Tyn=0}
Let $(X,d,W)$ be a $UCW$-hyperbolic space with a monotone  modulus of uniform convexity $\eta$, $C\se X$ a nonempty convex subset, and $T:C\rightarrow C$ nonexpansive with $F(T)\ne\emptyset$.

If $\ds\sum_{n=0}^\infty\lambda_n(1-\lambda_n)$ is divergent, then $\liminf_n d(x_n,Ty_n)=0$ for all $x\in C$.

Furthermore, if $\theta :\N\to\N$ is a rate of divergence for  $\ds\sum_{n=0}^\infty\lambda_n(1-\lambda_n)$, then for all  $x\in C,\eps>0,k\in\N$ there exists $N\in\N$ such that 
\beq 
k\leq N\leq h(\varepsilon,k,\eta,b,\theta) \text{~~and~~} d(x_N,Ty_N)<\varepsilon,\label{quant-liminf-d-xn-Tyn}
\eeq
 where
\[h(\varepsilon,k,\eta,b,\theta):=\left\{\begin{array}{ll}\displaystyle \theta\left(\left\lceil\frac{b+1}{\varepsilon\cdot\eta\left(b,\displaystyle\frac{\varepsilon}{b}\right)}\right\rceil+k\right)  & \text{for~ } \varepsilon \le 2b,\\
k & \text{otherwise,}
\end{array}\right.
\]
with $b>0$ such that $b\ge d(x,p)$ for some $p\in F(T)$.
\eprop
\begin{proof}
Let $x\in C, p\in F(T)$ and $b>0$ be such that $d(x,p)\leq b$. Since $(d(x_n,p))$ is decreasing, it follows that $d(x_n,p)\le d(x,p)\le b$ for all $n\in\N$.

Let $\eps>0, k\in\N$ and $\theta:\N\to\N$ be as in the hypothesis. We shall prove the existence of $N$ satisfying (\ref{quant-liminf-d-xn-Tyn}), which implies $\liminf_n d(x_n,Ty_n)=0$. 

First, let us remark that $d(x_n,Ty_n)\le 2d(x_n,p)\le 2b$ for all $n\in\N$, hence the case $\varepsilon > 2b$ is obvious. Let us consider $\varepsilon<2b$ and denote
\[P:=\left\lceil\frac{b+1}{\eps\eta\left(b,\frac{\eps}{b}\right)}\right\rceil,\]
so $h(\varepsilon,k,\eta,b,\theta):=\theta(P+k)\ge P+k> k$.  

Assume by contradiction that $d(x_n,Ty_n)\geq\varepsilon$ for all $n=\ol{k,\theta(P+k)}$. Since $\ds b\geq d(x_n,p)\geq \frac{d(x_n,Ty_n)}2\geq\frac{\eps}2$,   we can apply Lemma \ref{Ishikawa-main-technical-lemma}.(\ref{IMTL-eta}) with $\ds \beta:=\tilde{\beta}:=b, \gamma:=\frac{\eps}2$ and $a:=\eps$  to obtain that for all $n=\overline{k,\theta(P+k)}$,
\begin{eqnarray}
d(x_{n+1},p)&\le& d(x_n,p)- \eps\lambda_n(1-\lambda_n)\eta\left(b,\frac{\eps}{b}\right).\label{ineq-cor-1}
\end {eqnarray}
Adding (\ref{ineq-cor-1}) for $n=\overline{k,\theta(P+k)}$, it follows that 
\begin{eqnarray*}
d(x_{\theta(P+k)+1},p)&\le& d(x_k,p)-\eps\eta\left(b,\frac{\eps}{b}\right)\sum_{n=k}^{\theta(P+k)}\lambda_n(1-\lambda_n)\\
&\le& b-\eps\eta\left(b,\frac{\eps}{b}\right)\cdot P\leq b-(b+1)=-1,
\end {eqnarray*}
that is a contradiction. We have used the fact that 
\bua
 \sum_{n=k}^{\theta(P+k)}\lambda_n(1-\lambda_n)&= &\sum_{n=0}^{\theta(P+k)}\lambda_n(1-\lambda_n)-\sum_{n=0}^{k-1}\lambda_n(1-\lambda_n)\\
&\ge& \sum_{n=0}^{\theta(P+k)}\lambda_n(1-\lambda_n)-k\ge (P+k)-k=P.
\eua
\end{proof}

As an immediate consequence of the above proposition, we get a rate of asymptotic regularity for the Krasnoselski-Mann iterates, similar with the one obtained in \cite[Theorem 1.4]{Leu07}.

\bcor
Let $(X,d,W),\eta,C,T,b,(\lambda_n),\theta$ be as in the hypotheses of Proposition 
\ref{Ishikawa-liminf-xn-Tyn=0} and assume that $(x_n)$ is the Krasnoselski-Mann iteration starting with $x$, defined by (\ref{Ishikawa-KM-iteration-def-hyp}). 

Then $\limn d(x_n,Tx_n)=0$ for all $x\in C$ and, furthermore, 
\beq
\forall \eps>0\, \forall n\ge \Phi(\eps,\eta,b,\theta)\bigg(d(x_n,Tx_n)<\eps\bigg), \label{Ishikawa-KM-rate-as-reg}
\eeq
where $\Phi(\varepsilon,\eta,b,\theta):=h(\varepsilon,0,\eta,b,\theta)$,
with $h$ defined as above.
\ecor
\begin{proof}
Applying Proposition \ref{Ishikawa-liminf-xn-Tyn=0} with $s_n:=0$ (hence $y_n=x_n$) and $k:=0$, we get the existence of $N\le \Phi(\varepsilon,\eta,b,\theta)$ such that $d(x_N,Tx_N)<\eps$. Use the fact that $(d(x_n,Tx_n))$ is decreasing to get (\ref{Ishikawa-KM-rate-as-reg}).
\end{proof}

\bprop\label{Ishikawa-liminf-xn-Txn=0}
In the hypotheses of the above proposition, assume moreover that $\limsup_n s_n<1$. Then $\liminf_n d(x_n,Tx_n)=0$ for all $x\in C$.
 
Furthermore, if $L,N_0\in\N$ are such that $\ds s_n\leq 1-\frac1L$ for all $n\geq N_0$, then for all  $x\in C,\eps>0,k\in\N$ there exists $N\in\N$ such that 
\beq
k\leq N\leq \Psi(\varepsilon,k,\eta,b,\theta,L,N_0) \text{~~and~~} d(x_N,Tx_N)<\varepsilon, \label{quant-Ishikawa-liminf-d-xn-Txn}
\eeq
where $\ds \Psi(\varepsilon,k,\eta,b,\theta,L,N_0):=h\left(\frac{\eps}L, k+N_0,\eta,b,\theta\right)$, with $h$ defined as in Proposition \ref{Ishikawa-liminf-xn-Tyn=0}.
\eprop
\begin{proof}
Let $x\in C,\eps>0, k\in\N$. Applying Proposition \ref{Ishikawa-liminf-xn-Tyn=0} for $k+N_0$ and $\ds\frac{\eps}L$, we get the existence of $N$ such that $\ds N_0\le k+N_0\le N\le h\left(\frac{\eps}L, k+N_0,\eta,b,\theta\right)=\Psi(\eps,k,\eta,b,\theta,L,N_0)$ and $\ds d(x_N,Ty_N)<\frac{\eps}L$. 
Using (\ref{d-xn-Txn-d-xn-Tyn}) and the hypothesis, it follows that 
\bua
d(x_N,Tx_N)&\le & \frac{1}{1-s_N}d(x_N,Ty_N) < \frac{L\eps}L=\eps.
\eua
\end{proof}

As a corollary, we obtain an approximate fixed point bound for the nonexpansive mapping $T$.

\bcor\label{Ishikawa-AFP-bound}
In the hypotheses of Proposition \ref{Ishikawa-liminf-xn-Txn=0},
\beq
\forall \eps>0\, \exists N\le \Phi(\varepsilon,\eta,b,\theta,L,N_0)\bigg( d(x_N,Tx_N)<\varepsilon\bigg),
\eeq
where $\Phi(\eps,\eta,b,\theta,L,N_0):=\Psi(\eps,0,\eta,b,\theta,L,N_0)$ with $\Psi$ defined as above.
\ecor

We are ready now to prove the main result of this section. 

\bthm\label{Ishikawa-quant-as-reg}
Let $(X,d,W)$ be a $UCW$-hyperbolic space with a monotone  modulus of uniform convexity $\eta$, $C\se X$  a nonempty convex subset, and $T:C\rightarrow C$ nonexpansive with $F(T)\ne\emptyset$.\\
Assume that $\ds\sum_{n=0}^\infty\lambda_n(1-\lambda_n)$ diverges, $\limsup_n s_n<1$ and $\ds\sum_{n=0}^\infty s_n(1-\lambda_n)$ converges. 

Then for all $x\in C$,
\[\limn d(x_n,Tx_n)=0.\] 
Furthermore, if $\theta$ is a rate of divergence for $\ds\sum_{n=0}^\infty\lambda_n(1-\lambda_n)$, $L,N_0$ are such that $\ds s_n\leq 1-\frac1L$ for all $n\geq N_0$  and $\gamma$ is a Cauchy modulus for $\ds\sum_{n=0}^\infty s_n(1-\lambda_n)$, then for all  $x\in C$,
\beq
\forall \eps>0\forall n\ge \Phi(\eps,\eta,b,\theta,L,N_0,\gamma)\bigg(d(x_n,Tx_n)<\eps\bigg),
\eeq
where 
\[\Phi(\varepsilon,\eta,b,\theta,L,N_0,\gamma):= \left\{\begin{array}{ll}\!\!\displaystyle \theta\left(\left\lceil\frac{2L(b+1)}{\varepsilon\cdot\eta\left(b,\displaystyle\frac{\varepsilon}{2Lb}\right)}\right\rceil+\gamma\left(\frac{\eps}{8b}\right)+N_0+1\right)  & \!\!\!\text{for~} \eps\le 4Lb,\\
\!\!\ds \gamma\left(\frac{\eps}{8b}\right)+N_0+1 & \!\!\!\text{otherwise,}
\end{array}\right.
\]
with $b>0$ such that $b\ge d(x,p)$ for some $p\in F(T)$.
\ethm
\begin{proof}
Let $x\in C, p\in F(T)$ and $b>0$ be such that $d(x,p)\leq b$ and let us denote $\alpha_n:=\ds\sum_{i=0}^n s_i(1-\lambda_i)$. 
Since  $d(x_n,Tx_n)\le 2d(x_n,p)\le 2b$ for all $n\in\N$, we get by
(\ref{Ishikawa-as-reg-base-ineq}) that for all $n\in\N$,   
\bua
d(x_{n+1},Tx_{n+1})\leq (1+2s_n(1-\lambda_n))d(x_n,Tx_n)\leq d(x_n,Tx_n)+4bs_n(1-\lambda_n),
\eua
hence for all $m\in\N,n\ge 1$,
\bua
d(x_{m+n},Tx_{m+n})&\leq & d(x_n,Tx_n)+4b(\alpha_{n+m-1}-\alpha_{n-1}).
\eua
Let $\eps>0$ and apply Proposition \ref{Ishikawa-liminf-xn-Txn=0} with $\ds\frac{\eps}2$ and $\ds k:=\gamma(\eps/8b)+1$ to get $N\in\N$ such that $\ds d(x_N,Tx_N)<\frac{\eps}2$ and
\bua
\gamma(\eps/8b)+1\le N&\le& \Psi\left(\frac{\eps}2,\gamma(\eps/8b)+1,b,\theta,L,N_0\right)\\
&=& h\left(\frac{\eps}{2L}, \gamma(\eps/8b)+1+N_0,\eta,b,\theta\right)\\
&=&\Phi(\varepsilon,\eta,b,\theta,L,N_0,\gamma).
\eua
Since $\gamma$ is a Cauchy modulus for $(\alpha_n)$, it follows that for all $m\in\N$,
\[\alpha_{\ds m+\gamma(\eps/{8b})}-\alpha_{\ds\gamma(\eps/{8b})}=\left|\alpha_{\ds m+\gamma(\eps/{8b})}-\alpha_{\ds\gamma(\eps/{8b})}\right|<\frac{\eps}{8b}.\]

Let now $n\ge\Phi(\varepsilon,\eta,b,\theta,L,N_0,\gamma)\ge N$, hence $n=N+p=\gamma(\eps/8b)+1+q$ for some $p,q\in\N$. It follows that
\bua
d(x_n,Tx_n)&=&d(x_{N+p},Tx_{N+p})\le d(x_N,Tx_N)+4b(\alpha_{N+p-1}-\alpha_{N-1})\\
&=&d(x_N,Tx_N)+4b\left(\alpha_{\gamma(\eps/8b)+q}-\alpha_{N-1}\right)\\
&<& \frac{\eps}2+4b(\alpha_{\gamma(\eps/8b)+q}-\alpha_{\gamma(\eps/8b)})\\
&&\text{since~} N-1\ge \gamma(\eps/8b), \text{~so~} \alpha_{N-1}\ge\alpha_{\gamma(\eps/8b)}\\
&<& \frac{\eps}2+4b\cdot\frac{\eps}{8b}=\eps,
\eua
since $\gamma$ is a Cauchy modulus for $(\alpha_n)$.
\end{proof}

\begin{fact}\label{Ishikawa-quant-as-reg-tilde-eta}
In the hypotheses of Theorem \ref{Ishikawa-quant-as-reg}, assume, moreover, that  $\eta(r,\varepsilon)$ can be written as $\eta(r,\varepsilon)=\varepsilon\cdot\tilde{\eta}(r,\varepsilon)$ such that $\tilde{\eta}$ increases with $\varepsilon$ (for a fixed $r$). Then the bound $\Phi(\varepsilon,\eta,b,\theta,L,N_0,\gamma)$ can be replaced  for $\eps\le 4Lb$ with
\[ \tilde{\Phi}(\varepsilon,\eta,b,\theta,L,N_0,\gamma)=\theta\left(\left\lceil\frac{L(b+1)}{\varepsilon\cdot\tilde{\eta}\left(b,\displaystyle\frac{\varepsilon}{2Lb}\right)}\right\rceil+\gamma\left(\frac{\eps}{8b}\right)+N_0+1\right).\] 
\end{fact}
\begin{proof}
As we have seen in the proof of Theorem \ref{Ishikawa-quant-as-reg}, 
\[\Phi(\varepsilon,\eta,b,\theta,L,N_0,\gamma)=h\left(\frac{\eps}{2L}, \gamma\left(\frac{\eps}{8b}\right)+1+N_0,\eta,b,\theta\right),\]
where $h$ is defined as in Proposition \ref{Ishikawa-liminf-xn-Tyn=0}. It is easy to see that using the extra assumptions on $\eta$, $h(\varepsilon,k,\eta,b,\theta)$ can be replaced  for $\eps<2b$ with
\[ \tilde{h}(\varepsilon,k,\eta,b,\theta):= \theta\left(\left\lceil
\frac{b+1}{2\varepsilon\cdot \tilde{\eta}\left(b,\frac{\varepsilon}{b}\right)}\right\rceil+k\right).\] 
Just define $\ds P := \left\lceil\frac{b+1}{2\varepsilon\cdot \tilde{\eta}\left(b,\frac{\varepsilon}{b}\right)}\right\rceil$ and follow the proof of Proposition \ref{Ishikawa-liminf-xn-Tyn=0} using Lemma \ref{Ishikawa-main-technical-lemma}.(\ref{IMTL-tilde-eta}) (with $\delta:=b,a:=\eps$) instead of Lemma \ref{Ishikawa-main-technical-lemma}.(\ref{IMTL-eta}).
\end{proof}

\begin{corollary}\label{Ishikawa-bounded-C}
 Let $(X,d,W)$ be a complete $UCW$-hyperbolic space, $C\se X$  a nonempty convex closed bounded subset with finite diameter $d_C$ and $T:C\rightarrow C$ nonexpansive.\\
Assume that $\eta,(\lambda_n),(s_n),\theta, L,N_0,\gamma$ are as in the hypotheses of Theorem \ref{Ishikawa-quant-as-reg}. 

Then $\limn d(x_n,Tx_n)=0$ for all $x\in C$ and, moreover,
\[\forall \varepsilon >0\,\forall n\ge \Phi(\eps,\eta,d_C,\theta,L,N_0,\gamma)\,
\bigg(d(x_n,Tx_n) <\varepsilon\bigg), \]
where $\Phi(\eps,\eta,d_C,\theta,L,N_0,\gamma)$ is defined as in Theorem \ref{Ishikawa-quant-as-reg} by replacing $b$ with $d_C$.
\ecor
\begin{proof}
We can apply Corollary \ref{BGK-uc-hyp-monotone} to get that $F(T)\ne\emptyset$. Moreover, $d(x,p)\le d_C$ for any $x\in C, p\in F(T)$, hence we can take $b:=d_C$ in Theorem \ref{Ishikawa-quant-as-reg}.
\end{proof}

Thus, for bounded $C$, we get an  effective rate of asymptotic regularity which depends on the error $\varepsilon$, on the modulus of uniform convexity $\eta$,  on the diameter $d_C$ of $C$, on $(\lambda_n), (s_n)$ via $\theta,L,N_0,\gamma$, but does not depend  on the nonexpansive mapping $T$, the starting point $x\in C$ of the iteration or other data related with $C$ and $X$.

The rate of asymptotic regularity can be further simplified in the case of constant $\lambda_n:=\lambda\in(0,1)$.

\begin{corollary}\label{Ishikawa-bounded-C-constant-lambda}
 Let $(X,d,W),\eta,C,d_C,T$ be as in the hypotheses of Corollary \ref{Ishikawa-bounded-C}.
Assume that $\lambda_n:=\lambda\in(0,1)$ for all $n\in\N$. \\
Furthermore, let $L,N_0$ be such that $\ds s_n\leq 1-\frac1L$ for all $n\geq N_0$  and assume that the series  $\ds\sum_{n=0}^\infty s_n$ converges  with Cauchy modulus $\delta$. 

Then for all $x\in C$, 
\beq
\forall \eps>0\forall n\ge \Phi(\eps,\eta,d_C,\lambda,L,N_0,\delta)\bigg(d(x_n,Tx_n)<\eps\bigg),
\eeq
where 
\[\Phi(\varepsilon,\eta,d_C,\lambda,L,N_0,\delta):= \left\{\begin{array}{ll}\!\displaystyle \left\lceil\frac{1}{\lambda(1-\lambda)}\cdot \frac{2L(d_C+1)}{\varepsilon\cdot\eta\left(d_C,\displaystyle\frac{\varepsilon}{2Ld_C}\right)}\right\rceil+M & \!\!\!\text{for~ } \varepsilon \le 4Ld_C,\\
\! M & \!\!\!\text{otherwise,}
\end{array}\right.
\]
with $\ds M:=\delta\left(\frac{\eps}{8d_C(1-\lambda)}\right)+N_0+1$.

Moreover, if $\eta(r,\varepsilon)$ can be written as $\eta(r,\varepsilon)=\varepsilon\cdot\tilde{\eta}(r,\varepsilon)$ such that $\tilde{\eta}$ increases with $\varepsilon$ (for a fixed $r$),  then the bound $\Phi(\varepsilon,\eta,d_C,\lambda,L,N_0,\delta)$ can be replaced  for $\eps\le 4Ld_C$ with
\[
\Phi(\varepsilon,\eta,d_C,\lambda,L,N_0,\delta)=  \left\lceil\frac{1}{\lambda(1-\lambda)}\cdot \frac{L(d_C+1)}{\varepsilon\cdot\tilde{\eta}\left(d_C,\displaystyle\frac{\varepsilon}{2Ld_C}\right)}\right\rceil+M.
\]
\ecor
\begin{proof}
It is easy to see that  
\[\displaystyle \theta:\N\to\N, \quad \theta(n)=\left\lceil\frac{n}{\lambda(1-\lambda)}\right\rceil\]
is a rate of divergence for $\ds \sum_{n=0}^\infty\lambda(1-\lambda)$. Moreover, 
\[\gamma:(0,\infty)\to\N, \quad \gamma(\eps)=\delta\left(\frac{\eps}{1-\lambda}\right)\]
is a Cauchy modulus for $\ds\sum_{n=0}^\infty s_n(1-\lambda)$.
Apply now Corollary \ref{Ishikawa-bounded-C} and Remark \ref{Ishikawa-quant-as-reg-tilde-eta}. 
\end{proof}

As we have seen in Section \ref{Ishikawa-section-UC-hyp}, $CAT(0)$-spaces are $UCW$-hyperbolic spaces with a modulus of uniform convexity $\ds \eta(r,\varepsilon):=\frac{\varepsilon^2}{8}$, which has the form required in Remark \ref{Ishikawa-quant-as-reg-tilde-eta}. Thus, the above result can be applied to $CAT(0)$-spaces. 

\begin{corollary}\label{CAT0-constant-lambda}
In the hypotheses of Corollary \ref{Ishikawa-bounded-C-constant-lambda}, assume moreover that $X$ is a $CAT(0)$-space.

Then for all $x\in C$, 
\beq
\forall \eps>0\forall n\ge \Phi(\eps,d_C,\lambda,L,N_0,\delta)\bigg(d(x_n,Tx_n)<\eps\bigg),
\eeq
where 
\[\Phi(\varepsilon,d_C,\lambda,L,N_0,\delta):= \left\{\begin{array}{ll}
\left\lceil\ds \frac{D}{\eps^2}\right\rceil+M,
& \text{for~ } \varepsilon \le 4Ld_C,\\
M & \text{otherwise,}
\end{array}\right.
\]
with $\ds M:=\delta\left(\frac{\eps}{8d_C(1-\lambda)}\right)+N_0+1,\,\, D=\ds\frac{16L^2d_C(d_C+1)}{\lambda(1-\lambda)}$.
\ecor

\end{document}